\RequirePackage{rotating}
\documentclass[12pt]{amsart}
\usepackage{amsaddr, color, graphicx, hyperref, wasysym, marvosym, rotating, afterpage}
\usepackage[smalltableaux]{ytableau}
\renewcommand{\S}{\mathcal S}
\newcommand{\T}{\mathcal T}

\newcommand{\geo}{\mathrm{geo}}

\newtheorem{theorem}{Theorem}
\newtheorem{lemma}{Lemma}
\newtheorem{corollary}{Corollary}
\theoremstyle{definition}
\newtheorem{definition}{Definition}
\theoremstyle{example}
\newtheorem{example}{Example}
\theoremstyle{remark}
\newtheorem{remark}{Remark}

\title[Comparison of Gelfand-Tsetlin Bases]{Comparison of Gelfand-Tsetlin Bases for Alternating and Symmetric Groups}
\author{T. Geetha}
\address{Institute of Algebra and Number Theory\\University of Stuttgart\\Pfaffenwaldring 57, 70569 Stuttgart, Germany.}
\author{Amritanshu Prasad}
\address{The Institute of Mathematical Sciences (HBNI)\\CIT campus Taramani, Chennai 600113, India.\\\Letter{} \texttt{amri@imsc.res.in}\\\phone{} +91-44-22543207\\\Fax{} +91-44-22541586.}
\begin{document}
\maketitle
\begin{abstract}
  Young's orthogonal basis is a classical basis for an irreducible representation of a symmetric group.
  This basis happens to be a Gelfand-Tsetlin basis for the chain of symmetric groups.
  It is well-known that the chain of alternating groups, just like the chain of symmetric groups, has multiplicity-free restrictions for irreducible representations.
  Therefore each irreducible representation of an alternating group also admits Gelfand-Tsetlin bases.
  Moreover, each such representation is either the restriction of, or a subrepresentation of, the restriction of an irreducible representation of a symmetric group.
  In this article, we describe a recursive algorithm to write down the expansion of each Gelfand-Tsetlin basis vector for an irreducible representation of an alternating group in terms of Young's orthogonal basis of the ambient representation of the symmetric group.
  This algorithm is implemented with the Sage Mathematical Software.
\end{abstract}

\section{Introduction}
\label{sec:introduction}
The pioneers of representation theory, Frobenius, Schur and Young, studied the symmetric groups as a coherent sequence rather than in isolation.
Vershik and Okounkov \cite{Vershik2005} demonstrated that their results can be recovered by using information about the restrictions of irreducible representations of $S_n$ to $S_{n-1}$ as a starting point.

In their approach, Vershik and Okounkov established, from first principles, that every irreducible representation of $S_n$, when restricted to $S_{n-1}$, has a multiplicity-free decomposition.
They used this to construct a commutative subalgebra of the group algebra of $S_n$, known as the Gelfand-Tsetlin algebra, from whose characters they constructed the irreducible representations of $S_n$, and also determined the \emph{branching rules}, namely, which irreducible representations of $S_{n-1}$ occur in the restriction of an irreducible representation of $S_n$.
Besides their original article, notable expositions of the Vershik-Okounkov approach are found in \cite{ceccherini2010representation} and \cite{Muralinotes}.

In the Vershik-Okounkov approach, the irreducible representations of $S_n$ are constructed in terms of canonical bases, called \emph{Gelfand-Tsetlin bases}, whose vectors are uniquely determined up to scaling (for more details, see Section~\ref{sec:gt-bases}).
It turns out that Young's classical seminormal and orthogonal bases are examples of Gelfand-Tsetlin bases.
The property of symmetric group representations that allows us to construct Gelfand-Tsetlin bases is that every representation of $S_n$, when restricted to $S_{n-1}$, has a multiplicity-free decomposition into irreducible representations of $S_{n-1}$.

This property also holds for the alternating groups $A_n$.
Oliver Ruff~\cite{doi:10.1142/S1005386708000382} showed how to prove this from first principles without any reference to the representation theory of symmetric groups.
He then proceeded to construct the Gelfand-Tsetlin algebras associated to alternating groups.
By analyzing the characters of these algebras, he was able to construct their irreducible representations with respect to Gelfand-Tsetlin bases, and also obtain the branching rules.
All of this was achieved without appealing to the representation theory of symmetric groups.

The goal of this article is to determine the relationship between Gelfand-Tsetlin bases of the symmetric groups and the alternating groups.
It has been known since the days of Frobenius \cite{frobenius1901charaktere} that every irreducible representation of $A_n$ is either the restriction of an irreducible representation of $S_n$, or one of two non-isomorphic equidimensional subrepresentations of such a restriction (see also \cite[Section~5.1]{fulton2013representation} and \cite[Section~4.6]{rtcv}).
It therefore makes sense to ask how a Gelfand-Tsetlin basis vector for a representation of $A_n$ may be expanded in terms of a Gelfand-Tsetlin basis for a representation of $S_n$.
In this article we explain how to obtain the coefficients in such an expansion.
We use Young's orthogonal basis as the Gelfand-Tsetlin basis for an irreducible representation of $S_n$.

\section{Gelfand-Tsetlin bases}
\label{sec:gt-bases}

Let $\{G_n\}_{n=1}^\infty$ be an increasing sequence of finite groups with $G_1$ the trivial group.
Assume that, for each $n>1$, the restriction of each irreducible complex representation of $G_n$ to $G_{n-1}$ is multiplicity-free.

Let $\{V_\lambda\mid \lambda\in \Lambda_n\}$ denote a complete set of irreducible complex representations of $G_n$.
For $\lambda\in \Lambda_n$ and $\mu\in \Lambda_{n-1}$ write $\mu\in \lambda^\downarrow$ if $V_\mu$ occurs in the restriction of $V_\lambda$ to $G_{n-1}$.
Furthermore, denote the $V_\mu$ isotypic subspace of $V_\lambda$ by $V_{(\mu,\lambda)}$.
By Schur's lemma, the subspaces in the decomposition of the restriction of $V_\lambda$ to $G_{n-1}$
\begin{equation*}
  V_\lambda = \bigoplus_{\mu \in \lambda^\downarrow} V_{(\mu,\lambda)}
\end{equation*}
are uniquely determined \cite{Vershik2005}.
Furthermore, the copy $V_{(\mu,\lambda)}$ of $V_\mu$ in the right hand side of the above sum again has a decomposition as representations of $G_{n-2}$:
\begin{displaymath}
  V_{(\mu,\lambda)} = \bigoplus_{\nu\in \mu^\downarrow} V_{(\nu,\mu,\lambda)},
\end{displaymath}
where $V_{(\nu,\mu,\lambda)}$ is the $V_\nu$ isotypic subspace of $V_{(\mu,\lambda)}$.
We have
\begin{displaymath}
  V_\lambda = \bigoplus_{\nu\in \mu^\downarrow}\bigoplus_{\mu\in \lambda^\downarrow} V_{(\nu, \mu,\lambda)}.
\end{displaymath}
Continuing in this manner all the way down to $G_1$ gives rise to a unique decomposition
\begin{equation}
  \label{eq:13}
  V_\lambda = \bigoplus_{T\in \T(\lambda)} V_T,
\end{equation}
where $\T(\lambda)$ denotes the set of all sequences $T = (\lambda^{(1)},\dotsc, \lambda^{(n)})$ with $\lambda^{(i)}\in \lambda^{(i+1)\downarrow}$ for each $i=1,\dotsc, n-1$ and $\lambda^{(n)} = \lambda$.
We shall refer to (\ref{eq:13}) as the \emph{Gelfand-Tsetlin decomposition} of $V_\lambda$.
Moreover, $V_T$ (being isomorphic to the trivial representation of $G_1$) is a one dimensional subspace of $V_\lambda$ for each $T\in \T(\lambda)$.

Choosing a non-zero vector $x_T$ from $V_T$ for each $T\in \T(\lambda)$ gives rise to a basis:
\begin{displaymath}
  \{x_T\mid T\in \T(\lambda)\}
\end{displaymath}
of $V_\lambda$ which is known as a \emph{Gelfand-Tsetlin basis}.

We record, for later use, the following recursive description of $V_T$:
\begin{lemma}
  \label{lemma:recursive}
  Let $\lambda\in \Lambda_n$, and $T = (\lambda^{(1)},\dotsc, \lambda^{(n)})\in \T(\lambda)$.
  Then a vector $x\in V_\lambda$ lies in $V_T$ if and only if
  \begin{enumerate}
  \item $x$ lies in the $V_{\lambda^{(n-1)}}$ isotypic subspace $V_{(\lambda^{(n-1)}, \lambda)}$ of $V_\lambda$, and
  \item under an isomorphism $\phi: V_{\lambda^{(n-1)}}\tilde\to V_{(\lambda^{(n-1)}, \lambda)}$, $x$ lies in the image of $V_{(\lambda^{(1)},\dotsc,\lambda^{(n-1)})}$.
  \end{enumerate}
\end{lemma}
\begin{proof}
  Let $\mu\in \lambda^\downarrow$.
  Restricting the summands in the Gelfand-Tsetlin decomposition (\ref{eq:13}) of $V_\lambda$ to those $T=(\mu^{(1)},\dotsc,\mu^{(n-1)})\in \T(\lambda)$ for which $\mu^{(n-1)}=\lambda^{(n-1)}$ gives rise to the Gelfand-Tsetlin decomposition of $V_{(\lambda^{(n-1)},\lambda)}$.
  Thus the one dimensional subspace in the Gelfand-Tsetlin decomposition of $V_{(\lambda^{(n-1)}, \lambda)}$ corresponding to $(\lambda^{(0)},\dotsc,\lambda^{(n-1)})\in T(\lambda^{(n-1)})$ is the subspace $V_T$ of $V_\lambda$.
  Since isotypic components are preserved under isomorphism, so is the Gelfand-Tsetlin decomposition, and the lemma follows.
\end{proof}
Let $\Lambda = \coprod_{n=1}^\infty \Lambda_n$.
Define a graph structure on $\Lambda$ by connecting $\lambda$ to $\mu$ by an edge if and only if $\mu\in \lambda^\downarrow$.
The resulting graph is called the  \emph{Bratteli diagram} of the sequence $\{G_n\mid n \geq 1\}$.

When $G_n=S_n$, then $\Lambda_n$ is the set of partitions of $n$.
Recall that a partition of $n$ is a sequence $\lambda = (\lambda_1,\dotsc,\lambda_l)$ of positive integers in weakly decreasing order with sum $n$.
For $\lambda\in \Lambda_n$ and $\mu\in\Lambda_{n-1}$, we have $\mu\in \lambda^\downarrow$ if $\mu$ can be obtained from $\lambda$ by reducing one of its parts by one (or deleting a part that is equal to $1$).
Partitions are visually represented by their Young diagrams.
The \textit{Young diagram} for $\lambda = (\lambda_1,\dotsc,\lambda_l)$ is a left-justified array of boxes, with $\lambda_i$ boxes in the $i$th row.
We have $\mu\in\lambda^\downarrow$ if the Young diagram of $\mu$ is obtained from the Young diagram of $\lambda$ by removing one box.

The resulting Bratteli diagram is Young's graph (see Figure~\ref{fig:young}).
\begin{figure}[h]
  \centering
  \includegraphics[width=\textwidth]{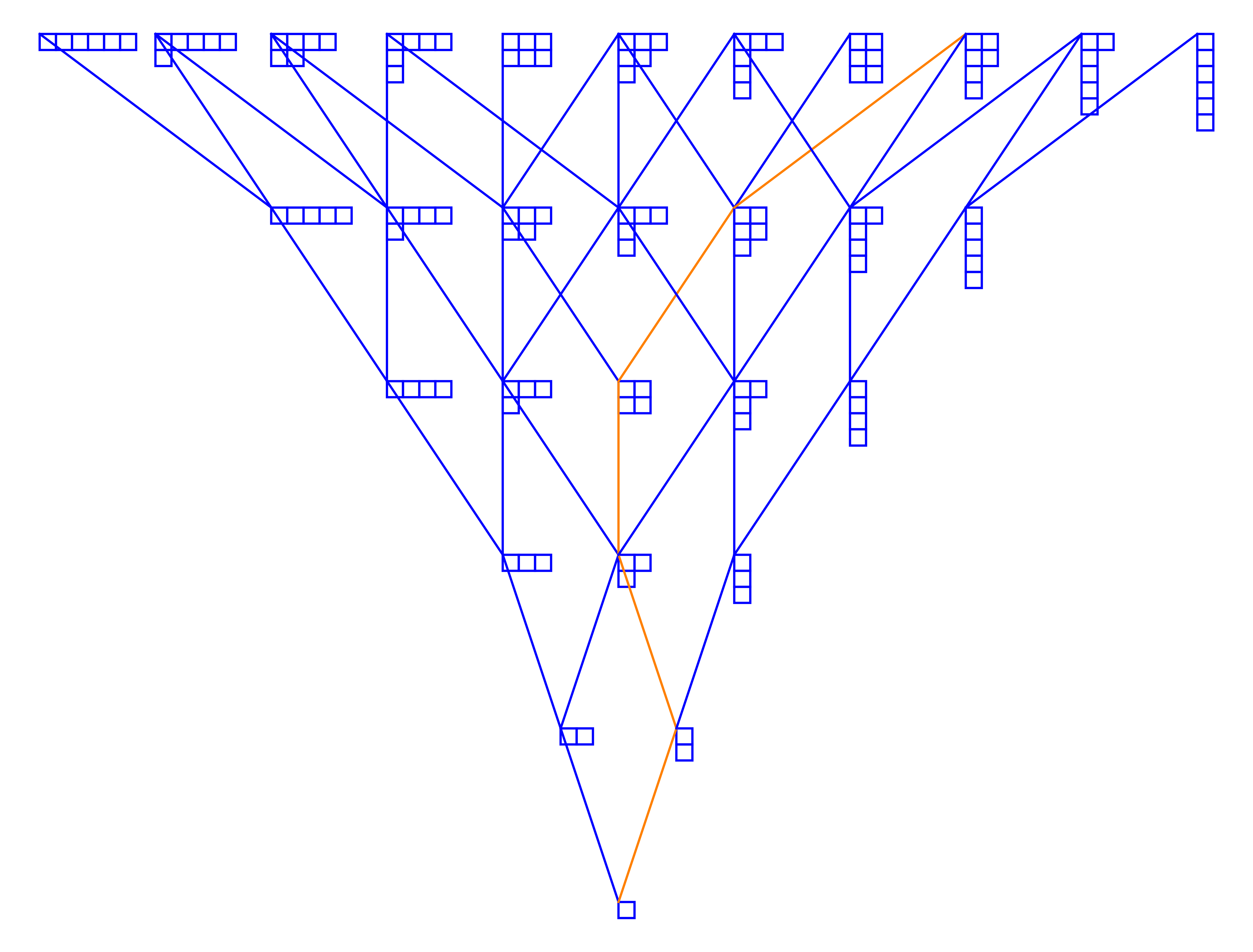}
  \caption{The first six rows of Young's graph}
  \label{fig:young}
\end{figure}
Given a partition $\lambda$, an element of $\T(\lambda)$ is nothing but a geodesic connecting the partition $(1)$ of $1$ to $\lambda$ in Young's graph.
Such a path can be represented by a \textit{standard Young tableau}.
This is a filling-in of the boxes of the Young diagram of $\lambda$ by integers $1,\dotsc, n$ increasing along rows and columns.
The Young diagram of the $i$th node in the geodesic is the set of boxes with numbers $1,\dotsc, i$ in the corresponding tableau.
For example, the path represented by orange edges in Figure~\ref{fig:young} is represented by the Young tableau
\begin{displaymath}
  \ytableaushort{13,24,5,6}.
\end{displaymath}
A particular normalization of the Gelfand-Tsetlin basis vectors results in Young's orthogonal basis of $V_\lambda$:
\begin{displaymath}
  \{v_T\mid T\in \T(\lambda)\},
\end{displaymath}
for which Young \cite{young1932quantitative} gave the following explicit formulas to determine the action of the symmetric group on $V_\lambda$ (see \cite[Theorem~2.1]{headley1996young}):
Let $s_i$ denote the $i$th simple transposition $(i, i+1)$ in $S_n$.
Let $T$ be a standard Young tableau of shape $\lambda$.
Suppose that $i$ appears in the $j$th row and $k$th column of $T$, and $i+1$ appears in the $j'$th row and $k'$th column of $T$.
Define $r = (k'-j') - (k-j)$.
Then the action of $s_i\in S_n$ on $V_\lambda$ is given by:
\begin{equation}
  \label{eq:4}
  \rho_\lambda(s_i)v_T =
  \begin{cases}
    v_T & \text{if $j=j'$,}\\
    -v_T & \text{if $k=k'$,}\\
    r^{-1}v_T + \sqrt{1-r^{-2}}v_{s_iT} &\text{otherwise}.
  \end{cases}
\end{equation}

In the last of the three cases above, since $i$ and $i+1$ lie in neither the same row nor the same column of $T$, interchanging their positions in $T$ results in a new standard Young tableau, which we have denoted by $s_iT$.

Throughout this paper, $\{v_T\mid T\in \T(\lambda)\}$ will denote Young's orthogonal basis of $V_\lambda$.

Suppose that $\mu\in \lambda^\downarrow$ and $T = (\lambda^{(1)},\dotsc,\lambda^{(n-1)})\in \T(\mu)$.
Define an element $T^\lambda\in \T(\lambda)$ by
\begin{equation}
  \label{eq:12}
  T^\lambda = (\lambda^{(1)},\dotsc,\lambda^{(n-1)},\lambda).
\end{equation}
It is clear from Young's formulae (\ref{eq:4}) that the linear map defined by
\begin{equation}
  \label{eq:9}
  v_T\mapsto v_{T^\lambda}.
\end{equation}
is an embedding of $V_\mu$ into the restriction of $V_\lambda$ to $S_{n-1}$.
Throughout this paper, whenever $\mu\in\lambda^\downarrow$, we shall identify $V_\mu$ with a subspace of $V_\lambda$ via this embedding.

\section{From symmetric groups to alternating groups}
\label{sec:from-symm-groups}

For $n\geq 2$, the symmetric group has one non-trivial multiplicative character; namely the sign character $\epsilon:S_n\to \{\pm 1\}$.
The kernel of this character is the alternating group $A_n$.
For each irreducible representation $\rho_\lambda$ of $S_n$, the representation $\epsilon\otimes \rho_\lambda:w\mapsto \epsilon(w)\rho_\lambda(w)$ is again an irreducible representation of $S_n$.
It turns out that the representation $\epsilon\otimes\rho_\lambda$ of $S_n$ is isomorphic to $\rho_{\lambda'}$, the representation corresponding to the partition $\lambda'$ that is conjugate to $\lambda$ \cite[Theorem~4.4.2]{rtcv}.
At the level of Young diagrams, conjugation is the reflection of a diagram about the principal diagonal.
For example, take $\lambda = (2^2, 1^2)$ and consider the reflection of its Young diagram about its principal diagonal:
\begin{displaymath}
  \ydiagram{2,2,1,1}\quad \longrightarrow \quad \ydiagram{4,2}.
\end{displaymath}
So $\lambda'$ is the partition $(4,2)$.
By Schur's lemma, there exists a unique (up to scaling) linear $S_n$-isomorphism $\phi_\lambda:\epsilon\otimes\rho_\lambda\to \rho_{\lambda'}$ which satisfies:
\begin{equation}
  \label{eq:2}
  \rho_{\lambda'}(s_i)\phi_\lambda + \phi_\lambda\rho_\lambda(s_i) = 0 \text{ for } 1\leq i \leq n-1.
\end{equation}
Following Headley \cite[Theorem~2.2]{headley1996young} and Stembridge \cite[Section~6]{stembridge1989eigenvalues}, we call $\phi_\lambda$ an \textit{associator} corresponding to $V_\lambda$.
We now recall the explicit computation of $\phi_\lambda$ from the identities (\ref{eq:4}) and (\ref{eq:2}) following \cite{thrall1941young,headley1996young}:
\begin{lemma}
  \label{lemma:conj}
  For every partition $\lambda$ and for every standard tableau $T$ of shape $\lambda$, the associator $\phi_\lambda:V_\lambda\to V_{\lambda'}$ satisfies
  \begin{displaymath}
    \phi_\lambda(v_T) = c_Tv_{T'},
  \end{displaymath}
  for some scalar $c_T$. Here $T'$ denotes the standard tableau conjugate to $T$ (the reflection of $T$ about its principal diagonal).
\end{lemma}
\begin{proof}
  We proceed by induction on $n$, the sum of the parts of $\lambda$.
  When $n\leq 2$, then $V_\lambda$ is always one dimensional, and the result follows trivially.
  Now suppose that $n>2$ and the result holds for $n-1$.
  Since edge relations in Young's graph are preserved by conjugation, the restriction of $V_{\lambda'}$ to $S_{n-1}$  is given by
  \begin{displaymath}
    V_{\lambda'} = \bigoplus_{\mu \in \lambda^\downarrow} V_{\mu'}.
  \end{displaymath}
  Since $\epsilon\otimes \rho_{\mu'}$ is isomorphic to $\rho_\mu$ as a representation of $S_{n-1}$,  an associator $\phi_\lambda:V_\lambda\to V_{\lambda'}$ restricts to an associator $\phi_\mu: V_\mu\to V_{\mu'}$.
  For each $T = (\lambda^{(1)}, \dotsc,\lambda^{(n)})\in \T(\lambda)$, let $T^\downarrow$ denote the element $(\lambda^{(1)},\dotsb,\lambda^{(n-1)})\in \T(\lambda^{(n-1)})$.
  If $T^\downarrow$ has shape $\mu$, then $v_T$ is the image of $v_{T^\downarrow}$ in $V_\lambda$ under the embedding defined by (\ref{eq:9}).
  Thus $\phi_\lambda(v_T) = \phi_\mu(v_{T^\downarrow})$.
  By induction $\phi_\mu(v_{T^\downarrow}) = c_{T^\downarrow}v_{{T^\downarrow}'}$.
  Since the image of $v_{T^{\downarrow\prime}}$ in $V_{\lambda'}$ under the embedding defined by (\ref{eq:9}) is $v_{T'}$, we have $\phi_\lambda(v_T) = c_Tv_{T'}$ for some constant $c_T$.
\end{proof}
\begin{corollary}
 For every partition $\lambda$, the associator $\phi_\lambda$ maps a GT-basis to a GT-basis.
\end{corollary}

Lemma~\ref{lemma:conj} reduces the calculation of $\phi_\lambda$ to the determination of the constants $c_T$ for each $T\in \T(\lambda)$.
The identity (\ref{eq:2}) is equivalent to
\begin{equation}
  \label{eq:5}
  \rho_{\lambda'}(s_i)(\phi_\lambda(v_T)) + \phi_\lambda(\rho_\lambda(s_i)(v_T)) = 0 \text{ for } 1\leq i \leq n-1, \text{ and all } T\in \T(\lambda).
\end{equation}
If $i$ and $i+1$ lie in the same row of $T$, they lie in the same column of $T'$.
Using (\ref{eq:4}), we have
\begin{align*}
  \rho_\lambda(s_i)(\phi_\lambda(v_T)) + \phi_\lambda(\rho_\lambda(s_i)(v_T)) & =
  \rho_\lambda(s_i)(c_Tv_{T'}) + \phi_\lambda(v_T)\\
  & = -c_Tv_{T'} + c_Tv_{T'}\\
  & = 0.
\end{align*}
A similar calculation shows that (\ref{eq:5}) is satisfied if $i$ and $i+1$ occur in the same column of $T$.
Now consider the case where $i$ and $i+1$ are neither in the same column, nor in the same row of $T$.
Applying (\ref{eq:4}) and using the fact that when $T$ is replaced by $T'$, then $r$ becomes $-r$, the left hand side of (\ref{eq:5}) simplifies to
\begin{multline*}
  \rho_\lambda(s_i)(c_Tv_{T'}) + \phi_\lambda(r^{-1}v_T + \sqrt{1-r^{-2}} v_{s_iT}) \\= c_T(-r^{-1}v_{T'} + \sqrt{1-r^{-2}}v_{s_iT'})  + r^{-1}c_Tv_{T'} + \sqrt{1-r^{-2}}c_{s_iT}v_{s_iT'}\\
  = (c_T + c_{s_iT})\sqrt{1-r^{-2}} v_{s_iT'}.
\end{multline*}
Thus, (\ref{eq:5}) is satisfied if and only if we can find a system $\{c_T\mid T\in \T(\lambda)\}$ of scalars such that, for all $T\in \T(\lambda)$ and all $i$ such that $i$ and $i+1$ do not lie in the same row or same column of $T$,
\begin{displaymath}
  c_{s_iT} = -c_T.
\end{displaymath}
For each partition $\lambda$ fix a standard tableau $T_\lambda$ of shape $\lambda$.
For each tableau $T\in \T(\lambda)$ let $w_T$ be the permutation which takes each entry of $T_\lambda$ to the corresponding entry of $T$.
Then, if we fix $c_{T_\lambda}$ arbitrarily, and define
\begin{equation}
  \label{eq:6}
  c_T = c_{T_\lambda} \epsilon(w_T),
\end{equation}
then $\phi_\lambda$ satisfies (\ref{eq:2}).

When $\lambda$ is self-conjugate, then one may additionally require that $\phi_\lambda^2 = \mathrm{id}_{V_\lambda}$.
The identity $\phi_\lambda^2(v_{T_{\lambda}}) = v_{T_\lambda}$ corresponds to the condition:
\begin{displaymath}
  c_{T_\lambda'}c_{T_\lambda} = 1,
\end{displaymath}
which, by (\ref{eq:6}), reduces to
\begin{displaymath}
  c_{T_\lambda}^2 = \epsilon(w_{T_\lambda'}).
\end{displaymath}
Therefore there are two possible choices of $\phi_\lambda$.
Note that $w_{T_\lambda'}$ is the permutation which interchanges each entry of $T_\lambda$ with the corresponding entry of $T_{\lambda'}$.
Its fixed points are the entries on the diagonal.
It is therefore a composition of $(n-d(\lambda))/2$ involutions, where
\begin{displaymath}
  d(\lambda) := \max\{i\mid i\leq \lambda_i\}
\end{displaymath}
is the length of the diagonal in the Young diagram of $\lambda$.
Hence the two possible choices for $\phi_\lambda$ are given by:
\begin{equation}
  \label{eq:11}
  \phi_\lambda(v_T) = \pm i^{(n-d(\lambda))/2}\epsilon(w_T)v_{T'}.
\end{equation}

Given a self-conjugate partition $\lambda$, set $d=d(\lambda)$.
Then $(d,d)$ is the last cell on the principal diagonal of the Young diagram of $\lambda$.
If $(d,d)$ is a removable cell, then removing it results in the unique self-conjugate partition $\mu\in \lambda^\downarrow$.
If $(d,d)$ is not removable, $(d+1,d+1)$ is an addable cell of $\lambda$, and adding it results in the unique self-conjugate partition $\mu$ such that $\lambda\in \mu^\downarrow$.
\begin{definition}
\label{definition:self-conj-cover}
  [Self-conjugate cover]
  A \textit{self-conjugate cover} is a pair $(\mu,\lambda)$ of self-conjugate partitions such that $\mu\in \lambda^\downarrow$.
\end{definition}
The pairs $((2,1), (2^2)$, $((3,1^2),(3,2,1))$, $((4,1^3),(4,2,1))$, and\linebreak $((3^2,2),(3^3))$ are some examples of self-conjugate covers.

Every self-conjugate partition is part of a unique self-conjugate cover; it is either the smaller or the larger partition in that pair.
We would like to arrange our choice of signs in (\ref{eq:11}) in such a way that when $(\mu,\lambda)$ is a self-conjugate cover, then $\phi_\lambda|_{V_\mu} = \phi_\mu$.
This is done by making judicious choices of $T_\lambda$.
As in Defintion~\ref{definition:self-conj-cover}, let $\mu$ denote the smaller partition, and $\lambda$ the larger partition, of a self-conjugate cover $(\mu, \lambda)$. Choose $T_\mu$ to be the ``row superstandard tableau'' of shape $\mu$, the tableau for which, when the rows are read from left to right and top to bottom, we get the integers $1,\dotsc,n$ in their natural order \cite[Section~A3.5]{bjorner2006combinatorics}.
Define $T_\lambda$ to be the tableau obtained by adding a box containing $n$ to the principal diagonal of $T_\mu$.
In the notation of (\ref{eq:12}), $T_\lambda = T_\mu^\lambda$.
For example,
\begin{displaymath}
  T_{(3, 1^2)} = \ytableaushort{123,4,5},\text{ and } T_{(3,2,1)} = \ytableaushort{123,46,5}.
\end{displaymath}

We now fix $\phi_\lambda$ by setting:
\begin{equation}
  \label{eq:1}
  \phi_\lambda(v_T) = i^{(n-d(\lambda))/2} \epsilon(w_T) v_{T'}.
\end{equation}
for every self-conjugate partition $\lambda$.
This is a more careful version of \cite[Theorem~V]{thrall1941young} (see also, \cite[Theorem~2.2]{headley1996young}), where $w_T$ is defined with respect to $T_\lambda$ specified above.

\begin{example}
  The factors $i^{(n-d(\lambda))/2}$ for all self-conjugate partitions $\lambda$ of integers less than $10$ are as follows:
  \Small{
  \begin{displaymath}
    \begin{array}{cccccccccc}
      \hline
      \lambda & (2, 1) & (2^2) & (3, 1^2) & (3, 2, 1) & (4, 1^3) & (4, 2, 1^2) & (3^2, 2) & (3^3) & (5, 1^4) \\
      \hline
      c_{T_\lambda} & i & i & -1 & -1 & -i & -i & -i & -i & 1\\
      \hline
     \end{array}
  \end{displaymath}
}
\end{example}
For each self-conjugate partition $\lambda$, define the representations $V_\lambda^+$ and $V_\lambda^-$ of $A_n$ as the eigenspaces for $\phi_\lambda$ with eigenvalues $+1$ and $-1$ respectively.
We have (see \cite[Section~5.1]{fulton2013representation} and \cite[Section~4.6]{rtcv}):
\begin{theorem}
  \label{theorem:an-classification-reps}
  For each $n\geq 2$, and for each $\lambda$ a partition of $n$,
  \begin{enumerate}
  \item if $\lambda\neq \lambda'$, then the restrictions of $V_\lambda$ and $V_{\lambda'}$ to $A_n$ are isomorphic irreducible representations of $A_n$
  \item if $\lambda = \lambda'$, then $V_{\lambda}^+$ and $V_\lambda^-$ are non-isomorphic irreducible representations of $A_n$.
  \end{enumerate}
  Moreover these cases give a complete set of representatives of isomorphism classes of irreducible representations of $A_n$.
\end{theorem}
\section{Branching rules for alternating groups}
The branching rules for the tower of alternating groups are easily deduced from those for symmetric groups.
The only issue that needs care is the behavior of the sign when branching from a self-conjugate partition to a self-conjugate partition, which is ultimately determined by our choice of normalization (\ref{eq:1}) of the associator.
\begin{lemma}
  Suppose that $\lambda$ is a self-conjugate partition, and that $\mu\in \lambda^\downarrow$ is also a self-conjugate partition.
  Then the restriction of the representation $V_\lambda^+$ of $A_n$ to $A_{n-1}$ contains $V_\mu^+$, and the restriction of $V_\lambda^-$ to $A_{n-1}$ contains $V_\mu^-$.
\end{lemma}
\begin{proof}
  Whenever $\lambda$ and $\mu\in \lambda^\downarrow$ are both self-conjugate, the Young diagram of $\lambda$ is obtained from the Young diagram of $\mu$ by adding a box on the diagonal.
  It follows that the number of boxes strictly below the diagonal in the Young diagram of $\lambda$ is the same as the number of boxes strictly below the diagonal in $\mu$, whence $c_{T_\lambda} = c_{T_\mu}$ (here $T_\lambda$ and $T_\mu$ are the tableaux chosen in the previous section).
Therefore, when $V_\mu$ is identified as a subspace of $V_\lambda$ via (\ref{eq:9}), we have $\phi_\lambda|_{V_\mu} = \phi_\mu$.
Since $V_\mu^\pm$ and $V_\lambda^\pm$ are the $\pm 1$ eigenspaces of $\phi_\mu$ and $\phi_\lambda$ respectively, the lemma follows.
\end{proof}

\begin{theorem}
  \label{theorem:an-branching}
  For a representation $W$ of $A_{n-1}$ and a representation $V$ of $A_n$, write $W\sqsubset V$ if $W$ occurs in the restriction of $V$ to $A_{n-1}$.
  Suppose that $\mu$ and $\lambda$ are partitions such that $\mu\in \lambda^\downarrow$ or $\mu'\in \lambda^\downarrow$.
  Then the relation ``$\sqsubset$'' is determined by the following table:
  \begin{displaymath}
    \begin{array}{|c|c|c|}
      \hline
      & \lambda \neq \lambda' & \lambda = \lambda'\\
      \hline
      \mu \neq \mu' & V_\mu \sqsubset V_\lambda & \begin{matrix}V_\mu \sqsubset V_\lambda^+\\ V_\mu \sqsubset V_\lambda^-\end{matrix}\\
      \hline
      \mu = \mu' & \begin{matrix}V_\mu^+ \sqsubset V_\lambda\\ V_\mu^- \sqsubset V_\lambda\end{matrix} & \begin{matrix}V_\mu^+ \sqsubset V_\lambda^+\\ V_\mu^- \sqsubset V_\lambda^-\end{matrix}\\
      \hline
    \end{array}
  \end{displaymath}
\end{theorem}
The Bratteli diagram for $A_n$, for $2\leq n\leq 9$, is shown in Figure~\ref{fig:bratelli}.
\afterpage{\clearpage}
\begin{sidewaysfigure}[h!]
  \includegraphics[width=\textwidth]{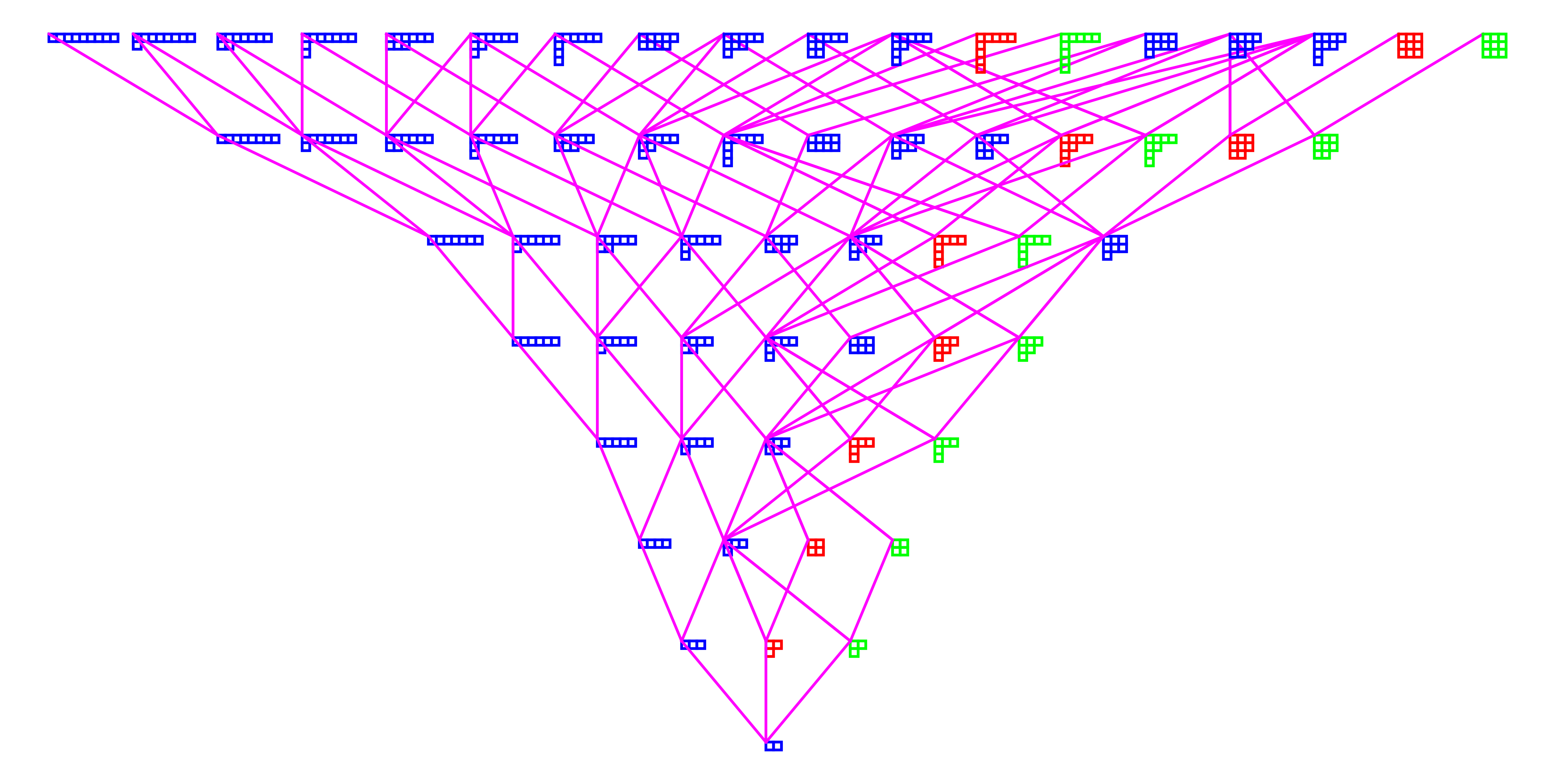}
  \caption{Bratteli diagram for alternating groups}
  \label{fig:bratelli}
\end{sidewaysfigure}

For self-conjugate partitions $\lambda$, $V_\lambda^+$ and $V_\lambda^-$ are distinguished by showing them in different colors, red and green respectively.
From each mutually conjugate pair $\{\lambda, \lambda'\}$ of non-self-conjugate partitions, we have chosen the one which comes first when all partitions are sorted in reverse lexicographic order.
\begin{remark}
  Formally, the branching rules given here coincide with those of Ruff \cite[Theorem~6.1]{doi:10.1142/S1005386708000382}.
  However, our careful normalizations of associators to identify $V_\lambda^+$ and $V_\lambda^-$ as subspaces of the representation $V_\lambda$ of $S_n$, for $\lambda$ self-conjugate, are new.
\end{remark}
\begin{remark}
  Note the edge joining $(4, 2^2)$ and $(3^2, 1)$ in Figure~\ref{fig:bratelli}.
  This is a valid edge even though $(3^2, 1)\notin (4,2^2)^\downarrow$.
  The conjugate of $(3^2,1)$ is $(3, 2^2)$, and $(3, 2^2)\in (4,2^2)^\downarrow$, so the representation $V_{(3^2,1)}$ of $A_7$, being isomorphic to $V_{(3,2^2)}$ occurs in the restriction of the representation $V_{(4,2^2)}$ of $A_8$ to $A_7$.
\end{remark}
\section{Geodesics for Alternating groups}
\label{sec:gt-alternating}

Let $\tilde\Theta(n)$ denote the set whose elements are of the form
\begin{enumerate}
\item $\lambda$, where $\lambda$ is a non-self-conjugate partition of $n$, or
\item $\lambda^\pm$, where $\lambda$ is a self-conjugate partition of $n$.
\end{enumerate}

For each $\alpha\in \tilde\Theta(n)$, define $V_\alpha$ as follows: if $\alpha$ is a non-self-conjugate partition, then $V_\alpha$ is the restriction of the irreducible representation of $S_n$ indexed by $\alpha$ to $A_n$.
If $\alpha = \lambda^\pm$, then $V_\alpha$ is the representation $V_\lambda^\pm$ of $A_n$ from Theorem~\ref{theorem:an-classification-reps}.
Define an equivalence relation on $\tilde\Theta(n)$, by setting $\alpha\sim \beta$ if either $\alpha=\beta$, or if both $\alpha$ and $\beta$ are non-self-conjugate partitions, and $\alpha = \beta'$.
Then the representations $V_\alpha$ and $V_\beta$ of $A_n$ are isomorphic if and only if $\alpha\sim \beta$.
The set $\Theta(n)$ of equivalence classes in $\tilde\Theta(n)$ is an index set for the set of isomorphism classes of irreducible representations of $A_n$.

Given $\alpha\in\tilde\Theta(n)$ and $\beta\in \tilde\Theta(n-1)$, say that $\beta\in \alpha^\dagger$ if one of the following occurs:
\begin{enumerate}
\item $\alpha$ and $\beta$ are non-self-conjugate partitions with $\beta\in \alpha^\downarrow$.
\item $\alpha$ is a non-self-conjugate partition, $\beta = \mu^\pm$ for a self-conjugate partition $\mu$, and $\mu\in \alpha^\downarrow$.
\item $\alpha = \lambda^\pm$ for some self-conjugate partition, $\beta$ is a non-self-conjugate partition, and $\beta\in \lambda^\downarrow$.
\item $\alpha = \lambda^+$, $\beta = \mu^+$ for self-conjugate partitions $\lambda$ and $\mu$ such that $\mu\in \lambda^\downarrow$.
\item $\alpha = \lambda^-$, $\beta = \mu^-$ for self-conjugate partitions $\lambda$ and $\mu$ such that $\mu\in \lambda^\downarrow$.
\end{enumerate}

For $\alpha\in \Theta(n)$, let $\tilde\S(\alpha)$ denote the set of all sequences of the form $\underline\alpha = (\alpha^{(2)},\dotsc, \alpha^{(n)})$, where $\alpha^{(i)}\in \tilde\Theta(i)$ for each $i$, $\alpha^{(i)}\in \alpha^{(i+1)\dagger}$ for $i=2,\dotsc, n-1$, and $\alpha^{(n)}= \alpha$.
Let $\geo(\alpha)$ denote the set of geodesic paths from $(2)$ to $\alpha$ in the Bratteli diagram of alternating groups.
By Theorem~\ref{theorem:an-branching}, each $\underline\alpha =(\alpha^{(2)},\dotsc,\alpha^{(n)})$ defines an element $\Gamma(\underline\alpha) = (V_{\alpha^{(2)}},\dotsc V_{\alpha^{(n)}}) \in \geo(\alpha)$.
Say that sequences $\underline\alpha = (\alpha^{(2)}, \dotsc, \alpha^{(n)})$ and $\underline\beta = (\beta^{(2)},\dotsc,\beta^{(n)})$ are equivalent, and write $\underline\alpha \sim \underline\beta$, if $\alpha^{(i)}\sim \beta^{(i)}$ for all $i=2,\dotsc,n$.
Thus $\Gamma(\underline\alpha) = \Gamma(\underline\beta)$ if and only if $\underline\alpha\sim\underline\beta$.

For each $\alpha\in \Theta(n)$, let $\S(\alpha)$ denote the set of equivalence classes in $\tilde\S(\alpha)$.
\begin{theorem}
  \label{theorem:parametrizing-paths}
  $\Gamma$ induces a bijection $\bar\Gamma:\S(\alpha)\to \geo(\alpha)$.
  Moreover, the cardinality of the class of $\underline\alpha = (\alpha^{(2)},\dotsc, \alpha^{(n)})$ in $\tilde\S(\alpha)$ is equal to $2^{r+1}$, where $r$ is the number of times $\underline\alpha$ ``branches from a self-conjugate partition to a non-self-conjugate partition'':
  \begin{multline*}
    r = \#\{2\leq i < n\mid \alpha^{(i)} =\mu^\pm \text{ for some } \mu = \mu',\\ \text{ and } \alpha^{(i+1)} = \lambda \text{ for some } \lambda \neq \lambda'\}.
  \end{multline*}
\end{theorem}
Before proving Theorem~\ref{theorem:parametrizing-paths}, we illustrate it with an example.
Consider $\underline\alpha \in \tilde\S(4,1^2)$ given by:
\begin{displaymath}
  \underline\alpha = ((2), (2, 1)^+, (3, 1), (3,1^2)^+, (4,1^2)).
\end{displaymath}
This sequence branches from self-conjugate to non-self-conjugate stages two times;
once from $(2,1)^+$ to $(3,1)$ and again from $(3,1^2)^+$ to $(4,1^2)$.
The eight elements of the class of $\underline\alpha$ in $\tilde\S(\alpha)$ are:
\begin{gather*}
  ((2), (2, 1)^+, (3, 1), (3,1^2)^+, (4,1^2))\\
  ((2), (2, 1)^+, (3, 1), (3,1^2)^+, (3,1^3))\\
  ((2), (2, 1)^+, (2, 1^2), (3,1^2)^+, (4,1^2))\\
  ((2), (2, 1)^+, (2, 1^2), (3,1^2)^+, (3,1^3))\\
  ((1^2), (2, 1)^+, (3, 1), (3,1^2)^+, (4,1^2))\\
  ((1^2), (2, 1)^+, (3, 1), (3,1^2)^+, (3,1^3))\\
  ((1^2), (2, 1)^+, (2, 1^2), (3,1^2)^+, (4,1^2))\\
  ((1^2), (2, 1)^+, (2, 1^2), (3,1^2)^+, (3,1^3))
\end{gather*}
\begin{proof}
  Theorem~\ref{theorem:parametrizing-paths} is proved by induction on $n$.
  For $n=2$, $r=0$, and $\tilde \S((2))$ has two elements, namely $((2))$ and $((1^2))$, so the theorem holds.

  Now suppose $\alpha\in \tilde\Theta(n)$.
  Every geodesic in $\geo(\alpha)$ is an extension of a unique geodesic $\gamma\in \geo(\beta)$ for some $\beta\in \alpha^\dagger$.
  If $\beta$ is of the form $\mu^\pm$ for some self-conjugate partition $\mu$ of $n-1$, and $\alpha$ is a non-self-conjugate partition, then $\mu\in \alpha^\dagger$, and also in $(\alpha')^\dagger$.
  By appending $\alpha$ or $\alpha'$ to $\gamma$, one obtains two distinct elements of $\tilde \S(\alpha)$ which extend $\gamma$.
  In all other cases, there is only one possible extension.
  The theorem now follows by induction.
\end{proof}

\section{Gelfand-Tsetlin bases for alternating groups}
\label{sec:gelf-tsetl-bases}

Given $\mu\in \lambda^\downarrow$, denote by $v^\lambda$ the image in $V_\lambda$ of $v\in V_\mu$ under the embedding defined by (\ref{eq:9}).
\begin{definition}
  Define
  \begin{equation}
    \label{eq:7}
    u_{(2)} = v_{\ytableaushort{12}},\quad u_{(1^2)} = v_{\ytableaushort{1,2}},
  \end{equation}
  Given $\underline\alpha = (\alpha^{(2)},\dotsc,\alpha^{(n)})\in \tilde \S(\alpha)$, let $\underline\alpha^*= (\alpha^{(2)},\dotsc,\alpha^{(n-1)})$.
  Recursively define $u_{\underline\alpha}$ as follows:
  If $\alpha$ is of the form $\lambda^\pm$ for a self-conjugate shape $\lambda$, and $\alpha^{(n-1)}$ is a non-self-conjugate partition of $n-1$, define $u_{\underline{\alpha}}\in V_\lambda$ by:
  \begin{equation}
    \label{eq:3}
    u_{\underline{\alpha}} = u_{\underline\alpha^*}^\lambda \pm \phi_\lambda(u_{\underline\alpha^*}^\lambda).
  \end{equation}
  If $\alpha=\alpha^{(n)}$ is of the form $\lambda^\pm$ for a self-conjugate $\lambda$, and $\alpha^{(n-1)} = \mu^{\pm}$ for self-conjugate $\mu\in \lambda^\downarrow$, define $u_{\underline{\alpha}}\in V_\lambda$ by:
  \begin{equation}
    \label{eq:8}
    u_{\underline{\alpha}} = u_{\underline\alpha^*}^\lambda.
  \end{equation}
  If $\alpha$ is a non-self-conjugate partition, then define
  \begin{equation}
    \label{eq:10}
    u_{\underline{\alpha}} = u_{\underline\alpha^*}^\alpha.
  \end{equation}
\end{definition}
\begin{example}
  Consider
  \begin{displaymath}
    \underline{\alpha} = ((2), (2, 1)^+, (3, 1), (3, 1^2)^-, (4, 1^2)).
  \end{displaymath}
  We compute recursively:
  \Small{
  \begin{align*}
    u_{((2),(2,1)^+)} & = v_{\ytableaushort{12,3}} + i v_{\ytableaushort{13,2}},\\
    u_{((2),(2,1)^+, (3,1))} & =  v_{\ytableaushort{124,3}} + i v_{\ytableaushort{134,2}},\\
    u_{((2),(2,1)^+, (3,1), (3,1^2)^-)} & =  v_{\ytableaushort{124,3,5}} + i v_{\ytableaushort{134,2,5}} - v_{\ytableaushort{135,2,4}} + i v_{\ytableaushort{125,3,4}},\\
    u_{((2),(2,1)^+, (3,1), (3,1^2)^-, (4, 1^2))} & =  v_{\ytableaushort{1246,3,5}} + i v_{\ytableaushort{1346,2,5}} - v_{\ytableaushort{1356,2,4}} + i v_{\ytableaushort{1256,3,4}}.
  \end{align*}
}
  Here we have used $\phi_{(2,1)}(v_{\ytableaushort{12,3}}) = iv_{\ytableaushort{13,2}}$, $\phi_{(3,1^2)}(v_{\ytableaushort{124,3,5}}) = v_{\ytableaushort{135,2,4}}$, and that $\phi_\lambda(v_{\ytableaushort{134,2,5}}) = -v_{\ytableaushort{125,3,4}}$.
\end{example}
\begin{theorem}
  [Main Theorem]
  For each $\alpha\in \tilde\Theta(n)$ and each geodesic $\gamma\in \geo(\alpha)$, if $\underline\alpha \in \tilde \S(\alpha)$ is such that $\Gamma(\underline\alpha) = \gamma$, then $u_{\underline\alpha}$ is a Gelfand-Tsetlin basis vector of $V_\alpha$ corresponding to $\gamma$.
  As $\underline\alpha$ runs over a set of representatives of paths in $\S(\alpha)$ all ending in $\alpha^{(n)}=\alpha$, a Gelfand-Tsetlin basis of $V_\alpha$ is obtained.
\end{theorem}
\begin{remark}
  If $\alpha$ is a non-self-conjugate partition, the class of a path $\underline\alpha$ in $\S(\alpha)$ contains paths with $\alpha^{(n)} = \alpha$ as well as paths with $\alpha^{(n)}\in \alpha'$.
  The condition that all the representatives $\underline\alpha$ have $\alpha^{(n)}=\alpha$ ensures that the basis vectors all lie in the same representation space $V_\alpha$.
\end{remark}
\begin{proof}
  The proof is by induction on $n$.
  The result holds for $n=2$ by definition (\ref{eq:7}).
  Suppose $\underline\alpha = (\alpha^{(2)},\dotsc,\alpha^{(n)})\in \tilde \S(\alpha)$.
  By Lemma~\ref{lemma:recursive}, in order to conclude that $u_{\underline\alpha}$ is a Gelfand-Tsetlin vector in $V_\alpha$ corresponding to $\Gamma(\underline\alpha)$, we need to show that it lies in the $V_{\alpha^{(n-1)}}$ isotypic subspace of $V_\alpha$,  and that, as a vector in that space, it is the isomorphic image of a Gelfand-Tsetlin vector corresponding to $\Gamma(\underline\alpha^*)$.
When $\alpha$ is a non-self-conjugate partition, this is clear.

Now suppose that $\alpha = \lambda^\pm$ for some self-conjugate partition $\lambda$ and also that $\alpha^{(n-1)}=\mu^\pm$ for some self-conjugate partition $\mu\in\lambda^\downarrow$.
Then the Young diagrams of $\lambda$ and $\mu$ differ by one cell, which lies on the principal diagonal.
Our definition (\ref{eq:1}) ensures that, the restriction of $\phi_\lambda$ to $V_\mu$ is $\phi_\mu$, so $u_{\underline\alpha}$ is an eigenvector for $\phi_\lambda$ with eigenvalue $+1$ if $\alpha = \lambda^+$, and with eigenvalue $-1$ if $\alpha=\lambda^-$.
Therefore $u_{\underline\alpha}\in V_\lambda^{\pm}$ as required.
By induction, $u_{\alpha^*}$ is a Gelfand-Tsetlin vector corresponding to $\alpha^*$, and the result follows by Lemma~\ref{lemma:recursive}.

Finally consider the case where $\alpha = \lambda^\pm$ for some self-conjugate partition $\lambda$ and $\alpha^{(n-1)}$ is a non-self-conjugate partition $\mu$ in $\lambda^\downarrow$.
Then (\ref{eq:3}) ensures that $u_{\underline\alpha}$ is an eigenvector for $\phi_\lambda$ with eigenvalue $\pm 1$, and hence lies in $V_\lambda^\pm$.
But when the representation $V_\lambda$ of $S_n$ is restricted to $S_{n-1}$, this vector is a sum of two vectors, one in $V_\mu$ and another in $V_{\mu'}$.
However, $u_{\underline\alpha}$ is the image under the isomorphism $\mathrm{id}_{V_\mu}\oplus \pm\phi_{\mu}$, from $V_\mu$ onto its image in $V_\mu\oplus V_{\mu'}$, of $u_{\underline\alpha^*}$, which is a Gelfand-Tsetlin vector corresponding to $\underline\alpha^*$ by induction.
\end{proof}
\begin{example}
  Representatives in $\tilde\S((4, 1^2))$ of the ten geodesics in \linebreak $\geo((4, 1^2))$ and the corresponding Gelfand-Tsetlin vectors are given by:
  \begin{enumerate}
  \item $\underline\alpha = ((1^2)$, $(1^3)$, $(2, 1^2)$, $(3, 1^2)^+$, $(4, 1^2))$, with
    \begin{displaymath}
      u_{\underline\alpha} = v_{\ytableaushort{1456,2,3}} - v_{\ytableaushort{1236,4,5}},
    \end{displaymath}
  \item $\underline\alpha = ((1^2)$, $(2, 1)^+$, $(2, 1^2)$, $(3, 1^2)^+$, $(4, 1^2))$, with
    \begin{displaymath}
      u_{\underline\alpha} = v_{\ytableaushort{1356,2,4}} + v_{\ytableaushort{1246,3,5}} - iv_{\ytableaushort{1256,3,4}} + iv_{\ytableaushort{1346,2,5}}.
    \end{displaymath}
  \item $\underline\alpha = ((1^2)$, $(2, 1)^-$, $(2, 1^2)$, $(3, 1^2)^+$, $(4, 1^2))$ with
    \begin{displaymath}
      u_{\underline\alpha} = v_{\ytableaushort{1356,2,4}} + v_{\ytableaushort{1246,3,5}} + iv_{\ytableaushort{1256,3,4}} - iv_{\ytableaushort{1346,2,5}}.
    \end{displaymath}
  \item $\underline\alpha = ((1^2)$, $(1^3)$, $(2, 1^2)$, $(3, 1^2)^-$, $(4, 1^2))$ with
    \begin{displaymath}
      u_{\underline\alpha} = v_{\ytableaushort{1456,2,3}} + v_{\ytableaushort{1236,4,5}}.
    \end{displaymath}
  \item $\underline\alpha = ((1^2)$, $(2, 1)^+$, $(2, 1^2)$, $(3, 1^2)^-$, $(4, 1^2))$ with
    \begin{displaymath}
      u_{\underline\alpha} = v_{\ytableaushort{1356,2,4}} - v_{\ytableaushort{1246,3,5}} - iv_{\ytableaushort{1256,3,4}} - iv_{\ytableaushort{1346,2,5}}.
    \end{displaymath}
  \item $\underline\alpha = ((1^2)$, $(2, 1)^-$, $(2, 1^2)$, $(3, 1^2)^-$, $(4, 1^2))$ with
    \begin{displaymath}
      u_{\underline\alpha} = v_{\ytableaushort{1356,2,4}} - v_{\ytableaushort{1246,3,5}} + iv_{\ytableaushort{1256,3,4}} + iv_{\ytableaushort{1346,2,5}}.
    \end{displaymath}
  \item $\underline\alpha = ((1^2)$, $(2, 1)^+$, $(3, 1)$, $(4, 1)$, $(4, 1^2))$ with
    \begin{displaymath}
      u_{\underline\alpha} = v_{\ytableaushort{1345,2,6}} - iv_{\ytableaushort{1245,3,6}}.
    \end{displaymath}
  \item $\underline\alpha = ((1^2)$, $(2, 1)^-$, $(3, 1)$, $(4, 1)$, $(4, 1^2))$ with
    \begin{displaymath}
      u_{\underline\alpha} = v_{\ytableaushort{1345,2,6}} + iv_{\ytableaushort{1245,3,6}}.
    \end{displaymath}
  \item $\underline\alpha = ((2)$, $(3)$, $(3, 1)$, $(4, 1)$, $(4, 1^2))$ with
    \begin{displaymath}
      u_{\underline\alpha} = v_{\ytableaushort{1235,4,6}}.
    \end{displaymath}
  \item $\underline\alpha = ((2)$, $(3)$, $(4)$, $(4, 1)$, $(4, 1^2))$ with
    \begin{displaymath}
      u_{\underline\alpha} = v_{\ytableaushort{1234,5,6}}.
    \end{displaymath}
  \end{enumerate}
\end{example}
This example was generated using Sage Mathematical Software~\cite{sage}.
The code is available from \url{http://www.imsc.res.in/\~amri/gt_alt.sage}.
\begin{remark}
  We chose to express the Gelfand-Tsetlin basis vectors of representations of $A_n$ in terms of Young's orthogonal basis because of the ease of computing associators in terms of this basis.
  The simple form of the formula (\ref{eq:6}) is a special property of Young's orthogonal basis.
  The bases obtained here depend on the choice of representatives in $\tilde S(\alpha)$ of geodesics in $\geo(\alpha)$.
  For distinct $\underline{\alpha}, \underline{\beta}\in \tilde S(\alpha)$ having $\Gamma(\underline{\alpha})=\Gamma(\underline{\beta})$, $u_{\underline{\alpha}}$ and $u_{\underline{\beta}}$ need not be equal.
  But being Gelfand-Tsetlin vectors for the same geodesic, these vectors are scalar multiples of each other.
  Since the number of (mutually orthogonal) terms in their expansions are the same, these scalars must be unitary.
  Since all the coefficients in the expansions are fourth roots of unity, we may conclude that these scalars are actually fourth roots of unity.
  For example:
  \begin{displaymath}
    u_{((2), (2,1)^+)} = v_{\ytableaushort{12,3}} + iv_{\ytableaushort{13,2}},
  \end{displaymath}
  while
  \begin{displaymath}
    u_{((1,1), (2,1)^+)} = -i v_{\ytableaushort{12,3}} + v_{\ytableaushort{13,2}}.
  \end{displaymath}
\end{remark}
\subsection*{Acknowledgements}
GT was supported by an Alexander von Humboldt Fellowship.
This paper was written while AP was visiting the University of Stuttgart.
He thanks Steffen K\"onig for supporting this visit.
He was also supported by a Swarnajayanti fellowship of the Department of Science \& Technology (India).
The authors thank the referees for their careful reading of this manuscript, and their suggestions for its improvement.
\bibliographystyle{abbrv}
\bibliography{refs}

\begin{thebibliography}{10}

\bibitem{bjorner2006combinatorics}
A.~Bjorner and F.~Brenti.
\newblock {\em Combinatorics of Coxeter Groups}.
\newblock Graduate Texts in Mathematics. Springer Berlin Heidelberg, 2006.

\bibitem{ceccherini2010representation}
T.~Ceccherini-Silberstein, F.~Scarabotti, and F.~Tolli.
\newblock {\em Representation Theory of the Symmetric Groups: The
  Okounkov-Vershik Approach, Character Formulas, and Partition Algebras}.
\newblock Cambridge Studies in Advanced Mathematics. Cambridge University
  Press, 2010.

\bibitem{frobenius1901charaktere}
F.~G. Frobenius.
\newblock {\"U}ber die charaktere der alternirenden gruppe, 1901.

\bibitem{fulton2013representation}
W.~Fulton and J.~Harris.
\newblock {\em Representation Theory: A First Course}, volume 129.
\newblock Springer Science \& Business Media, 2013.

\bibitem{headley1996young}
P.~Headley.
\newblock On {Y}oung's orthogonal form and the characters of the alternating
  group.
\newblock {\em Journal of Algebraic Combinatorics}, 5(2):127--134, 1996.

\bibitem{rtcv}
A.~Prasad.
\newblock {\em Representation Theory: A Combinatorial Viewpoint}.
\newblock Number 147 in Cambridge Studies in Advanced Mathematics. Cambridge
  University Press, Delhi, 2015.

\bibitem{doi:10.1142/S1005386708000382}
O.~Ruff.
\newblock Weight theory for alternating groups.
\newblock {\em Algebra Colloquium}, 15(03):391--404, 2008.

\bibitem{Muralinotes}
M.~K. Srinivasan.
\newblock Notes on the {V}ershik-{O}kounkov approach to the representation
  theory of the symmetric groups.
\newblock Available from \url{http://www.math.iitb.ac.in/~mks/}, 2007.

\bibitem{stembridge1989eigenvalues}
J.~Stembridge.
\newblock On the eigenvalues of representations of reflection groups and wreath
  products.
\newblock {\em Pacific Journal of Mathematics}, 140(2):353--396, 1989.

\bibitem{sage}
{T}he~{S}age {D}evelopers.
\newblock {\em {S}age {M}athematics {S}oftware ({V}ersion 7.1)}, 2016.
\newblock {\url{http://www.sagemath.org}}.

\bibitem{thrall1941young}
R.~Thrall et~al.
\newblock Young's semi-normal representation of the symmetric group.
\newblock {\em Duke Mathematical Journal}, 8(4):611--624, 1941.

\bibitem{Vershik2005}
A.~M. Vershik and A.~Y. Okounkov.
\newblock A new approach to the representation theory of the symmetric groups.
  {II}.
\newblock {\em Journal of Mathematical Sciences}, 131(2):5471--5494, 2005.

\bibitem{young1932quantitative}
A.~Young.
\newblock On quantitative substitutional analysis.
\newblock {\em Proceedings of the London Mathematical Society}, 2(1):196--230,
  1932.

\end{thebibliography}
\end{document}